\documentclass[11pt]{amsart}

\usepackage{mathtools}
\usepackage{amsmath}
\usepackage{amssymb}
\usepackage{yhmath}
\usepackage{graphicx}
\usepackage{ mathrsfs }
\usepackage{bbm}
\usepackage{xcolor}

\newtheorem{theorem}{Theorem}[section]

\newtheorem*{claim*}{Claim}

\newtheorem{lemma}[theorem]{Lemma}
\newtheorem{lem}[theorem]{Lemma}

\newtheorem{proposition}[theorem]{Proposition}

\newtheorem{prop}[theorem]{Proposition}

\newtheorem{thm}[theorem]{Theorem}

\theoremstyle{definition}

\theoremstyle{remark}

\numberwithin{equation}{section}

\newcommand{\norm}[1]{\lVert#1\rVert}

\newcommand{\op}{\operatorname}
\newcommand{\bb}{\mathbb}

\newcommand{\Ga}{\Gamma}

\newcommand{\ga}{\gamma}

\newcommand{\la}{\lambda}
\newcommand{\La}{\Lambda}
\newcommand{\ba}{\backslash}

\newcommand{\cl}{\overline}

\newcommand{\cal}{\mathcal}
\newcommand{\br}{\mathbb R}
\newcommand{\SO}{\op{SO}}

\newcommand{\bH}{\mathbb H}

\newcommand{\RFM}{\op{RF} {\cal S}_1}
\newcommand{\RFPM}{\op{RF}_+ {\cal S}_1}

\newcommand{\hull}{\op{hull}}

\newcommand{\be}{\begin{equation}}
\newcommand{\ee}{\end{equation}}

\newcommand{\G}{\Gamma}

\newcommand{\mS}{\mathscr{S}}

\newcommand{\T}{\op{T}}

\newcommand{\mT}{\mathsf T}
\renewcommand{\epsilon}{\varepsilon}

\newcommand{\bt}{{\bf t}}\newcommand{\bs}{{\bf s}}
\newcommand{\Om}{\Omega}\newcommand{\bS}{\mathbb S}
\DeclareMathAlphabet{\mathpzc}{OT1}{pzc}{m}{it}
\begin{document}

\title[Orbit closures  of Zariski dense subgroups]{Topological proof of Benoist-Quint's orbit closure theorem for $\SO(d,1)$}
\author{Minju Lee}

\address{Department of Mathematics, Yale University, New Haven, CT 06520}
\email{minju.lee@yale.edu}

\author{Hee Oh}
\address{Mathematics department, Yale university, New Haven, CT 06511}
\email{hee.oh@yale.edu}

\thanks{Oh was supported in part by NSF Grant \#1900101.}



\begin{abstract} { We present a new  proof of the following theorem of Benoist-Quint: 
Let $G:=\SO^\circ(d,1)$, $d\ge 2$ and $\Delta<G$  a cocompact lattice.
Any orbit of a Zariski dense subgroup $\Gamma$ of $G$ is either finite or dense in $\Delta\ba G$.
While Benoist and Quint's  proof is based on the classification of stationary measures, our proof is topological, using ideas from the study of dynamics of unipotent flows on
the infinite volume homogeneous space $\Gamma\ba G$. }
\end{abstract}

\maketitle

\section{Introduction}
Let $G=\SO^\circ (d,1)$ for $d\ge 2$, and 
 $\Delta$ a cocompact lattice in $G$.
  Let $\Ga$ be a Zariski dense subgroup of $G$, acting on the space $\Delta\ba G$ by right translations.

The aim of this paper is to present a new proof of the following theorem of Benoist-Quint in \cite{BQ1}, which was originally  a question of 
Margulis \cite{Ma} and Shah \cite{Sh0}:
\begin{thm}\label{mt} Any $\Ga$-invariant subset of $\Delta\ba G$ is either finite or dense.
\end{thm}

The proof of Benoist-Quint is based on their classification of stationary measures for random walks on $\Ga$ on the space $\Delta\ba G$.
Our proof  is topological and can be easily modified to all rank one simple Lie groups; for the sake of concreteness, we opted to write it only for $G=\SO^\circ (d,1)$.
 In the case
when $G=\SO^\circ(2,1)$ and $\Ga<G$ is a convex cocompact Zariski dense subgroup, Benoist-Oh gave a topological proof of Theorem \ref{mt} when the $\Gamma$-invariant subset is a single $\Gamma$-orbit \cite{BO}.  

Since a Zariski dense subgroup of $G$ is either discrete or dense, it suffices to consider the case when $\Gamma$ is discrete.
Our starting point is  then the observation that Theorem \ref{mt} can be translated into a problem on the orbit closure of unipotent flows
on a homogeneous space of {\it infinite volume}. If we set $H=\{(g,g): g\in G\}$ to be the diagonal embedding of $G$ into $G\times G$,
 Theorem \ref{mt} is equivalent to the following statement about
the $H$-action on the product space $\Gamma\ba G\times \Delta\ba G$, which has infinite volume unless $\Gamma$ is a lattice.
\begin{thm}\label{dm} 
 Any $H$-invariant  closed subset of $(\Gamma\times \Delta)\ba (G\times G )$ is either a union of finitely many closed $H$-orbits or dense.
In particular, any $H$-orbit is either closed or dense.
\end{thm}
When $\Ga$ is a lattice in $G$, i.e., when the homogeneous space
 $(\Gamma\times \Delta)\ba (G\times G )$ has finite volume, Theorem \ref{dm} is a special case of Ratner's orbit closure theorem  \cite{Ra2} and
  Mozes-Shah theorem \cite{MS2}.

  \medskip
  
\noindent{\bf On the proofs.} Any Zariski dense discrete subgroup of $G$ contains a Zariski dense Schottky subgroup
(Lemma \ref{sh}). Hence in proving Theorem \ref{mt}, we may assume without loss of generality that
$\Gamma$ is a convex cocompact Zariski dense subgroup.

  Set $Z:=(\Gamma\times \Delta)\ba (G\times G)$.
Let $\cal A=\{a_t\}$  be a one-parameter subgroup of diagonalizable elements of $G$, and  $\cal U$
  the contracting  horospherical subgroup of $G$ with respect to the choice of $\cal A$. 
  Let $U<H$ denote the diagonal embedding of $\cal U$ into $G\times G$.
Our proof is based on the study of the action of $U$ on $Z$.
Let $\Omega$ denote the subset of $Z$ consisting of all bounded $\cal A\times \cal A$-orbits, which  is a compact subset.
For $x\in \Om$,  consider the return of $xU$ to $\Omega$:
 $$\mT(x):=\{u\in U: xu\in \Om\}.$$
For any sequence $\la_i\to \infty$,  we show that
the renormalization  $$\mT_\infty:=\limsup_i \la_i^{-1} \mT(x)$$ is locally Zariski dense at $e$, i.e.,
for any neighborhood $\cal O$ of $e$ in $U$, $\mT_\infty\cap \cal O$ is Zariski dense in $U$ (Lemma \ref{ZD}).
This is the key recurrence property we use in carrying out the unipotent dynamics for the $U$-action on $Z$.
We remark that this recurrence property is much weaker than the notion of thickness used in (\cite{MMO}, \cite{MMO2}, \cite{BO2}, \cite{LO}),
where the thick return property was required  for {\it any} one-parameter subgroup of $U$; the latter strong  property does not hold
 for a general convex cocompact subgroup.

 
We prove that any closed $H$-invariant subset $X$ of $Z$, which is not a union of finitely many closed $H$-orbits, contains a $U$-minimal subset $Y$ with respect to $\Omega$ such that
$$Y\cal C\subset X$$
for some non-constant analytic curve $\cal C$ contained in $\{e\}\times \cal U$. 
We then conclude $X=Z$ using the density of translates  $x\cal Ca_{-t}\subset \Delta\ba G$
as $t\to +\infty$ (Theorem \ref{yang}); this last ingredient was proved by  Shah \cite{Sh2},
using Ratner's measure classification theorem \cite{Ra1} and the linearization techniques 
(\cite{DM}, \cite{Sh1}).\bigskip

\noindent{\bf Acknowledgement}
We would  like to thank the referee for a careful reading and helpful comments.

\section{Notations and background}
Let $G=\SO^\circ (d,1)$, $d\ge 2$.  Let $\bH^d$ denote the real hyperbolic space of dimension $d$ with boundary $\partial \bH^d=\mathbb S^{d-1}$.
 Then $G$ can be identified with the group $\op{Isom}^+(\bH^d)$ of orientation preserving isometries of $\bH^d$.
The isometric action of $G$ on $\bH^d$ extends to a transitive action of $G$ on the unit tangent bundle $\T^1(\bH^d)$.
We identify $\bH^d=G/\cal K$ and $\T^1(\bH^d)=G/\cal M$ where $\cal K$ and $\cal M$ are respectively
 the stabilizers of a point $o\in \bH^d$ and a vector $v_o\in \T^1_o(\bH^d)$. The group $G$ itself can be understood as the oriented frame bundle $\op{F}(\bH^d)$.
Let $\cal A=\{a_t: t\in \br\}$ be the one-parameter subgroup of diagonalizable elements such that
$\cal A$ centralizes $\cal M$ and the right translation action of $a_t$ on $G/\cal M$ corresponds to the geodesic flow on $\T^1(\bH^d)$.
For a tangent vector $v\in \T^1(\bH^d)$, we write $v^+$ for the forward end point of the associated geodesic in the boundary
$\mathbb S^{d-1}$ and $v^-$ for the backward end point. 
For $g\in G$, we define $$g^+:= (gv_o)^+=gv_o^+\quad\text{and}\quad g^-:= (gv_o)^- =gv_o^-.$$

We denote by $\cal U$ the contracting horospherical subgroup of $G$:
$$\cal U=\{u\in G: a_{-t} u a_t \to e,\quad \text{as $t\to +\infty$}\} .$$
The group $\cal U$ is isomorphic to $\br^{d-1}$; we write  $\cal U=\{u_{\bt}: {\bt}\in \br^{d-1}\}$.

We use the following notation in the rest of the paper:
\begin{itemize}
\item $H=\{(h,h) : h\in G\}$;
\item $H_1=G\times \{e\}$, $H_2=\{e\}\times G$,  and  $\cal H= H_1\times H_2$; 
\item $A=\{(a_t,a_t):t\in \br \}$;
\item $A_1=\cal A\times \{e\}$, $ A_2=\{e\}\times \cal A$;
\item $U= \{(u,u): u\in \cal U\}$;
\item $U_1=\cal U\times \{e\} $, $ U_2=\{e\}\times \cal U$;
\item $M=\{(m,m):m\in \cal M\}$;
\item $M_1=\cal M\times \{e\} $, and $ M_2=\{e\}\times \cal M$.
\end{itemize}

Let $\Ga_1<H_1$ be a Zariski-dense  discrete subgroup and $\Ga_2<H_2$ be a  cocompact lattice. We assume that both  $\Gamma_1$ and $\Gamma_2$
are torsion-free.
For each $i=1,2$, let $$\cal S_i:=\Ga_i\ba \bH^d$$ denote
 the associated real hyperbolic manifold, and  $\La_i\subset \mathbb S^{d-1}$ the limit set of $\Ga_i$. As $\G_2$ is a  lattice in $H_2$, we have $\La_2=\mathbb S^{d-1}$.
 We assume that $\Gamma_1$ is convex cocompact, that is, $\Gamma_1\ba \op{hull}(\Lambda_1)$ is compact where $\hull (\La_1)\subset \bH^d$
 denotes the convex hull of $\La_1$.

Set $$Z_1=\Gamma_1\ba H_1, \, \;\; Z_2=\Gamma_2\ba H_2 , \;\text{and} \;\;Z=Z_1\times Z_2.$$ 
We  define 
$$\RFM=\{x_1\in Z_1:x_1A_1\text{ bounded} \}=\{[g]\in Z_1: g^{\pm}\in \La_1\} ;$$
and
 $$\RFPM=\{x_1\in Z_1:x_1A^+_1\text{ bounded} \} =\{[g]\in Z_1: g^+\in \La_1\} $$
where $A_1^+=\{(a_s, e): s\ge 0\}$.

Define
$$\Omega=\RFM\times Z_2 \;\;\text{and}\; \; \Omega_+=\RFPM \times Z_2.$$ 

As $\Gamma_1$ is convex cocompact, $\RFM$  is a compact $A_1M_1$-invariant subset.
Hence $\Om$ is a compact subset of $Z$ which is invariant under $\prod_{i=1}^2 A_iM_i$.
The set $\RFPM$ is equal to $\RFM \cdot U_1$, and hence $\Om_+$ is  a closed subset of $Z$ invariant under $\prod_{i=1}^2 A_iM_i U_i$.

\section{Local Zariski density of renormalization of  $U$-recurrence}
We often identify $U$ with $\br^{d-1}$ via the map $(u_{\bt}, u_{\bt})\mapsto \bt$, and the notation $\|\bt \|$ means the Euclidean norm of $\bt\in \br^{d-1}$.
To ease the notation, we sometimes write $u\in U$, identifying $u$ with $(u,u)$.
Similarly we will write $a\in A$, identifying $a$ with $(a,a)$.

For $x\in \Omega$, we define the following recurrence time of $x$ to $\Omega$ under $U$:
$$\mT(x):=\{{\bt} \in \br^{d-1}: xu_{\bt} \in \Omega\}.$$

For $x=(x_1, x_2)\in \Om$, note that ${\bt}\in \mT(x)$ if and only if $x_1 u_{\bt} \in \RFM$.
If we choose $g_1\in H_1$ so that $x_1=[g_1]$, then $g_1^{\pm}\in \La_1$ since $x_1\in \RFM$.
Since $(g_1u_\bt )^+=g_1^+$, we have 
\be \label{mttt} \mT(x)=\{ {\bt} \in \br^{d-1}: (g_1 u_\bt)^-\in \La_1\}.\ee

Since $(g_1u_\bt)^-\to g_1^+\in \La_1$ as $\bt\to \infty$ and $\La_1$ has no isolated point, it follows that $\mT(x)$ is unbounded.

\begin{lem}\label{Zd}
For $x\in \Omega$, any non-empty open subset of $ \mT(x)$ is Zariski dense in $U$. \end{lem}
\begin{proof} The visual map $ U \to \bb S^{d-1}-\{g_1^+\}$ defined by $u \mapsto (g_1u)^-$
is a diffeomorphism. Hence by \eqref{mttt},  the claim follows
the well-known fact that no non-empty open subset of $\Lambda_1$ is 
contained in a smooth submanifold in $\bb S^{d-1}$ of positive co-dimension 
(cf.  \cite[Corollary 3.10]{Win}).
\end{proof}

\begin{lemma}\label{lem.ZD}\label{ZD}
Let  $x\in \Om$.  For any sequence  $\la_i\to + \infty$, there exists $z\in \Om$ such that
 $$\mT_\infty:=\limsup_{i\to\infty}\la_i^{-1}  { \mT(x)}  \supset \mathsf T(z).$$
In particular,  for any neighborhood $\cal O\subset U$ of $e$, $\mT_\infty \cap \cal O$ is  Zariski dense  in $U$.
\end{lemma}
\begin{proof} 
Note that  $\la_i^{-1} { \mT(x)}=\{{\bt}\in \br^{d-1}: x u_{\la_i \bt}\in \Om\}$.
Let $s_i=\tfrac{1}{2} \log \la_i$ so that $a_{s_i} u_{{\bt}}a_{s_i}^{-1}=u_{ \la_i {\bt}}$.
Since $\Om$ is $A$-invariant,
$$\la_i^{-1}  { \mT(x)} =\{{\bt}\in \br^{d-1} : xa_{s_i}  u_{ \bt}\in \Om\} =\mT(xa_{s_i}).$$
Since $\Om$ is a compact $A$-invariant subset, passing to a subsequence, $xa_{s_i}$ converges to some $z\in\Omega$ as $i\to\infty$.

We claim that
$\limsup_{i\to \infty} \la_i^{-1} { \mT(x)}
\supset \mT(z).$
Let $x=(x_1,x_2)$, $z=(z_1,z_2)$, and 
choose $g_1, g_1'\in G$ so that $x_1=[g_1]$ and $z_1=[g_1']$.
Since $x_1a_{s_i}\to z_1$ as $i\to \infty$, there exists $\ga_i\in\Ga_1$ such that $\ga_ig_1a_{s_i}\to g_1'$ as $i\to\infty$.
Let $\bt\in \mathsf{T}(z)$.
For each $i$, choose $\textbf{r}_i\in\bb{R}^{d-1}$ of minimal Euclidean norm in the set
$\{\textbf{r}\in\bb{R}^{d-1} : (\ga_ig_1a_{s_i}u_{\bt+\textbf{r}})^-\in\La_1\}$.
We claim that $\textbf{r}_i\to 0$ as $i\to\infty$.
Suppose not. Then there exists $c>0$ such that $B_i(c)\cap\La_1=\emptyset$ for infinitely many $i$, where $\op{B}(\bt,c)
\subset \bb{R}^{d-1}$ denotes the closed ball of radius $c$  centered at $\bt$, $u_{\op{B}(\bt,c)}=\{u_{\bold s}:\bold s\in \op{B}(\bt,c)\}$
and $B_i(c):=(\ga_ig_1a_{s_i}u_{\op{B}(\bt,c)})^-$.
Since $B_i(c)$  converges to $B(c):=(g_1'u_{\op{B}(\bt,c)})^-$ in the Hausdorff topology of closed subsets of $\bb{S}^{d-1}$
and $B(c)$ contains a neighborhood of $(g_1'u_{\bt})^-$, $B_i(c)$ must contain $(g_1'u_{\bt})^-$ for all  sufficiently large $i$.
Since  $(g_1'u_{\bt})^-\in\La_1$ (because $\bt\in \mathsf{T}(z)$), we get a contradiction to the hypothesis that $B_i (c)\cap\La_1=\emptyset$ for infinitely many $i$'s.

Hence $\bt_i:=\bt+\textbf{r}_i \to  \bt$, and  $\bt_i\in \la_i^{-1}\mathsf{T}(x)$ since
\begin{align*}
(\ga_ig_1u_{\la_i\bt_i})^-&=(\ga_ig_1a_{s_i}u_{\bt_i}a_{-s_i})^-=(\ga_ig_1a_{s_i}u_{\textbf{t}_i})^-\in\La_1.
\end{align*}
This shows that $\bt\in \limsup \la_i^{-1}\mathsf{T}(x)$,
proving the first claim.
The second claim follows from the first claim together with Lemma \ref{Zd}.  \end{proof}

\section{Unipotent blowup}
 For a subgroup $S<{\cal H}$, we denote by $\op{N}(S)$ the normalizer of $S$ in $\cal H$.
For a subgroup $S_i\subset H_i$,   $\op{C}_{H_i}(S_i)$ denotes the centralizer of $S_i$ in $H_i$.
\begin{lemma} We have
$\op{N}(U)=AMU_1U_2$.
\end{lemma}
\begin{proof}
The inclusion $\supset $ is clear.
To show the reverse inclusion $\subset$, let $(g_1,g_2)\in\op{N}(U)$.
Then   for all $(u,u)\in U$, $(g_1ug_1^{-1},g_2ug_2^{-1})\in U$ and hence $g_2^{-1}g_1ug_1^{-1}g_2=u$.
This implies $(g_2^{-1}g_1,e)\in\op{C}_{H_1}(U_1)$.
Since $\op{C}_{H_1}(U_1)\subset\op{N}(U)$, 
\begin{equation*}
(g_1,g_2)=(g_2,g_2)\cdot(g_2^{-1}g_1,e)\in\op{N}(U),
\end{equation*}
it follows $(g_2,g_2)\in\op{N}(U)\cap H =AMU$.
As both $(g_2,g_2)$ and $(g_2^{-1}g_1,e)$ belong to $AMU_1U_2$, so is $(g_1,g_2)$.
\end{proof}

\begin{lemma}\label{lem.UgU}
Let $g_i\to e$ in $\cal H-\op{N}(U)$ as $i\to\infty$, and
$x\in \Om$.  Then for any neighborhood $\cal O\subset \cal H$ of $e$,
 there exist sequences  $u_i'\in U$ and $u_i\in {}{\mathsf T(x)}$ such that, as $i\to \infty$,
the sequence $u_i'g_iu_i$ converges to an element in $(AMU_2-M)\cap \cal O$.

\end{lemma}
\begin{proof}
Following \cite{MT}, we will construct a quasi-regular map
\begin{equation*}
\psi : \br^{d-1} \to  AMU_2\end{equation*}
associated to the given  sequence $g_i$. 
 Since $U$ is a real algebraic subgroup of $\cal H$, by Chevalley's theorem, there exists an $\br$-regular representation ${{\cal H}}\to \op{GL}(W)$ with a distinguished point $p\in W$ such that $U=\op{Stab}(p)$. Then $p{{\cal H}}$ is locally closed, and $\op{N}(U)$ is equal to the set
\begin{equation}\label{eq.char1}
\{g\in {{\cal H}} : p gu =p g\text{ for all }u\in U\}.
\end{equation}

Set $L:=A_1M_1U_1^+\times A_2M_2U_2^+U_2$ where $U_i^+$ is the expanding horospherical subgroup of $H_i$ for $i=1,2$. Note that
$$\op{N}(U)\cap L=AMU_2$$
and that the product map from ${}{U}\times L$ to $\cal H$ is a diffeomorphism onto a Zariski open neighborhood
 of $e$. 
 
 Since $p{{\cal H}}$ is open in its closure and $pL$ is a open neighborhood of $p$ in $p{\cal H}$, we can choose an $M$-invariant norm on $W$
 such that
\begin{equation}\label{eq.m1}
B(p,1)\cap \cl{p{{\cal H}}}\subset p {{L}}
\end{equation}
 where $B(p, r)\subset W$ denotes the norm  ball of radius $r$ centered at $p$.

  For each $i$, we define
 $\tilde \phi_i: \br^{d-1}\to W$ by
 $$\tilde \phi_i(\bt )=p g_i u_{\bt} $$
 which is  a polynomial map in $d-1$-variables
with degree uniformly bounded for all $i$. Note that  $\tilde \phi_i (0)$ converges to $ p$ as $i\to \infty$.
As $g_i\not\in\op{N}(U)$, $\tilde \phi_i$ is non-constant.

Now define
\begin{equation*}
\la_i:=\sup\{\la\ge 0: \tilde \phi_i({}{B(\la)})\subset B(p,1)\}
\end{equation*}
where  $B(\la )$ denotes the norm ball of radius $\la$ centered at $0$ in $\br^{d-1}$.
Note that $\la_i<\infty$ as $\tilde \phi_i$ is nonconstant, and that $\la_i\to\infty$ as $g_i\to e$.
Reparametrizing  $\tilde \phi_i$ by $\la_i$, we define $\phi_i: \br^{d-1}\to W$:
$$\phi_i (\bt ):=\tilde \phi_i(\lambda_i \bt).$$

Note that $\sup\{ \|\phi_i (\bt) -p \|: \bt\in B(1)\}=1$, and $\lim_{i\to \infty} \phi_i(0)=p$.
Since the polynomials  $\phi_i$ have uniformly bounded degree, it follows that
 after passing to a subsequence,
 $\phi_i$ converges to a non-constant polynomial $\phi:\br^{d-1}\to W$ uniformly on every compact subset of $\br^{d-1}$.

Since $pL$ is a Zariski open neighborhood of $p$ in $\overline{p\cal H}$,
 the following map $\psi$ defines  a non-constant rational map  on a Zariski open  neighborhood
 of $0$ in $\br^{d-1}$:
 $$\psi:=\rho_L^{-1} \circ \phi$$
where $\rho_L$ is the restriction to $L$ of the orbit map $g\mapsto p g$.

Since $g_i\to e$, without loss of generality,  we may assume  that $g_i\in  U L$ for all $i$.
  Except for a Zariski closed subset of $\br^{d-1}$,
 the product $g_i u_\bt$ can be written as an element of $ {}{U}L$ in a unique way. 
 We denote by $\psi_i(\bt )\in L$ its $L$-component
 so that $g_i {}{u_{\bt}}\in U\psi_i (\bt).$
 
We have  $\psi(0)=e$ and
$$\psi(\bt)=\lim _{i \to \infty} \psi_i(\la _i \bt)$$
where  the convergence is uniform on compact subsets of $\br^{d-1}$. It is easy to check that
 $\op{Im} \psi \subset \op{N}(U)\cap L= AMU_2$ using \eqref{eq.char1}.
Set \begin{equation*}
\mathsf T_\infty := \limsup\limits_{i\to\infty} \la_i^{-1} {\mathsf T(x)}.
\end{equation*}

Given neighborhood $\cal O\subset \cal H$ of $e$,
 let $\cal O'$ be a neighborhood of $0$ in $\br^{d-1}$ such that $\phi(\cal O')\subset p \cal O$. Since $\phi$ is a nonconstant polynomial,
it follows from Lemma \ref{ZD} that
 there exists ${\bt}\in \cal O' \cap\mathsf{T}_\infty$ such that $\norm{\phi({\bt})}^2\neq\norm{p}^2$.

Let ${\bt}_i\in {}{\mathsf T(x)}$ be a sequence such that $\la_i^{-1} {\bt}_i \to {\bt}$ as $i\to\infty$
(by passing to a subsequence). Since $\psi_i\circ \la_i\to \psi$ uniformly on compact subsets,
\begin{equation*}
\psi({\bt})=\lim_{i\to\infty} (\psi_i\circ \la_i )\left(\la_i^{-1} {{\bt}_i} \right)=\lim_{i\to \infty} \psi_i({\bt}_i)
=\lim_{i\to\infty}u_{{\bs}_i}g_iu_{{\bt}_i}
\end{equation*}
for some sequence $\bs_i\in \br^{d-1}$.
Note that  $\phi({\bt})=p\psi({\bt})$ with $\psi({\bt})\in AMU_2\cap \cal O$. 
Since $\norm{\phi({\bt} )}^2\neq\norm{p}^2$ and $\|\cdot \|$ is $M$-invariant, we have $\psi({\bt})\not\in M$.
Hence this finishes the proof. \end{proof}


\begin{lemma}\label{lem.UgH}
Let $r_i\to e$ in $H_2 -\op{N}(U)$. For any $x\in \Om$,
there exists  a sequence $\bt_i\in \mathsf{T}(x)$  such that  the sequence $u_{-\bt_i} r_iu_{\bt_i} $
converges to a non-trivial  element of $ U_2$.
\end{lemma}
\begin{proof} 
Write $r_i=\exp q_i$ for $q_i\in \mathfrak{h}= \op{Lie}(H_2)$. We write $U_2=\{u_{\bt}:\bt\in \br^{d-1}\}$.
Define a polynomial map $\psi_i : \br^{d-1} \to \mathfrak{h}$ by
\begin{equation*}
\psi_i(\mathbf{t})= u_\mathbf{t}^{-1} q_i u_\mathbf{t} \quad\text{ for all $\mathbf{t}\in \br^{d-1}$}.
\end{equation*}
Since $H_2\cap\op{N}(U)=U_2=\op{C}_{H_2}(U_2)$, it follows that  $r_i\in H_2-\op{C}_{H_2}(U_2)$.
Hence $\psi_i $ is a nonconstant polynomial.
Let  $\la_i> 0$ be  the supremum of $\la>0$ such that
$\sup_{{\bt}\in {B (\la)}} {\|\psi_i({\bt}) \|} \le 1$ where $B(\la)$ denotes the
ball in $\br^{d-1}$ of radius $\la$ centered at $0$. Then $0<\la_i< \infty$ and $\la_i \to \infty$.

Now the rescaled polynomials 
 $\phi_i:=\psi_i\circ \la_i: \br^{d-1} \to \mathfrak{h}$
are uniformly bounded on the unit ball with uniformly bounded degree and
 $\lim_{i\to \infty} \phi_i(0)= 0$. Therefore, by passing to a subsequence,
$\phi_i$  converges
 to a  polynomial $\phi : \br^{d-1} \to \mathfrak{h}$ uniformly on compact subsets.
 Since $\sup_{{\bt}\in {B(1)}} {\|\phi({\bt}) \|} =1$, $\phi$ is not a constant.

We claim that $\op{Im}(\phi)\subset \op{Lie}(U_2)$.
For any fixed ${\bs}\in \br^{d-1}$, we have $\la_i^{-1} {\bs}\to 0$, and hence
for any ${\bt}\in \br^{d-1}$, 
\begin{align*} u_{\bs}^{-1} \phi({\bt}) u_{\bs} & =\lim_{i\to \infty} u_{-\la_i \bt -\bs} q_i  u_{\la_i {\bt} +\bs }\\
&=
\lim_{i\to \infty} u_{-\la_i ( \bt + \la_i^{-1} \bs)} q_i  u_{\la_i ({\bt} +\la_i^{-1} \bs )}
\\&=
\lim_{i\to \infty} u_{-\la_i \bt}  q_i  u_{\la_i {\bt}}
=\phi(\bt).
\end{align*}

Hence $ \phi({\bt})$ belongs to the centralizer of $U_2$.
Since the centralizer of $U_2$ in $\mathfrak h$ is equal to $\op{Lie}U_2$,  the claim follows.

Set
\begin{equation*}
\mathsf T_\infty:=\limsup\limits_{i\to\infty}\la_i^{-1} \mathsf T(x).
\end{equation*}
Fix $\mathbf{t}\in \mathsf T_\infty$  such that $\phi({\bt}) \ne 0$; this exists by Lemma \ref{ZD}.  Let ${{\bt}_i}\in \mathsf T(x)$  be a sequence
such that $\la_i^{-1} {{\bt}_i} \to \mathbf t$ as $i\to\infty$.
As $\phi_i\to \phi$ uniformly on compact subsets, it follows that
$$
\phi(\mathbf t)=\lim_{i\to \infty} (\psi_i \circ \la_i )(\la_i^{-1} {\bt}_i)= \lim_{i\to \infty} \psi_i ({\bt}_i) =
\lim_{i\to\infty} {u_{{\bt}_i}^{-1}}q_i {u_{{\bt}_i}}.
$$
Hence, by exponentiating, we obtain that  ${u_{{\bt}_i}^{-1}}r_i {u_{{\bt}_i}}$ converges to a non-trivial element of $U_2$.
\end{proof}

\section{Relative minimal subsets and additional invariance}

Let $X$ be a closed $H$-invariant subset of $Z$. 
A closed $U$-invariant subset $Y$ of $X$ is called $U$-minimal with respect to $\Om$ if $Y\cap \Om\ne \emptyset$ and for any $y\in Y\cap \Om$,
$yU$ is dense in $Y$.
Since every $H$-orbit in $Z$ intersects $\Omega$, $X\cap \Om \ne \emptyset$. By Zorn's lemma,
there exists a $U$-minimal subset $Y$ of $X$ with respect $\Om$, which we fix in the following.

\begin{lem}\label{min} If $\pi_i:Z\to Z_i$ denotes the canonical projection for $i=1,2$, we have
$$\pi_1(Y)=\RFPM\;\;\text{and }\;\; \pi_2(Y) =Z_2.$$
\end{lem}

\begin{proof}
The claim follows since $U_1$ and $U_2$ act minimally on $\RFPM$ and $Z_2$ respectively \cite{Win}.
\end{proof}

\begin{lemma}\label{lem.normalizerorbit} Let $S$ be a closed  subgroup of $\op{N}(U)$ containing $U$.
For any $y\in Y\cap \Om$, the orbit $yS$ is not locally closed.
\end{lemma}
\begin{proof} 
Suppose that $yS$ is locally closed for some $y\in Y\cap \Om$.
We claim that there exists  a sequence $u_i\to\infty$ in $U$ such that $yu_i\to y$ as $i\to \infty$.
Let 
\begin{equation*}
Q:=\{z\in Y : z=\lim\limits_{i\to\infty}yu_i\text{ for some }u_i\to\infty\text{ in }U\}.
\end{equation*}
Since $\mathsf{T}(y)$ is unbounded, there exists $u_i\to \infty$ in $U$ such that
$yu_i\in Y\cap \Om$. Since any limit of the sequence $yu_i$ belongs to $Q\cap \Omega$, 
we have $Q\cap{}\Om\neq\emptyset$.
Since $Q$ is a closed $U$-invariant set, $Q=Y$ by the relative $U$-minimality of $Y$.
In particular, $y\in Q$, proving the claim.
We may assume that $y=[e]$ without loss of generality.
Let $\Ga:=\Ga_1\times\Ga_2$.
Since $yS$ is locally closed, $yS$ is homeomorphic to the quotient $(S\cap\Ga)\ba S$.
Therefore there exists $\delta_i\in S\cap\Ga$ such that $\delta_iu_i\to e$ as $i\to\infty$.

Since $\op{N}(U)=AMU_1U_2$, writing $\delta_i=a_ir_i$ for $a_i\in A$ and $r_i\in MU_1U_2$, it follows that  $a_i\to e$ as $i\to\infty$.
Write $\delta_i=(\delta_i^1, \delta_i^2)\in \Ga_1\times \Ga_2$.
In the case when $a_i=e$ for all sufficiently large $i$,  it follows from $u_i\to \infty$ in $U$ that
$\delta_i^1$ must be a parabolic element of $\Gamma_1$, yielding a contradiction to the convex cocompactness of $\Ga_1$.
In the case when $a_i\ne e$ for an infinite subsequence,  we again get a contradiction, because  there is  a uniform  positive lower bound for all translation lengths of elements of $\Ga_1$. This finishes the proof.
\end{proof}
\begin{lemma}\label{lem.G-N} 
For any $y\in Y\cap \Om $, there exists a sequence  $g_i\to e$ in $\cal H-\op{N}(U)$ such that $yg_i\in Y$.
\end{lemma}
\begin{proof}
Suppose not. Then there is an open neighborhood $\cal O\subset \cal H$ of $e$ such that
\begin{equation}\label{eq.G-N} 
y\cal{O}\cap Y\subset y\op{N}(U).
\end{equation}
We may assume the map $g \mapsto yg\in X$ is injective on $\cal O$ by shrinking $\cal O$ if necessary.
 Set $$S:=\{g\in \op{N}(U) : Yg= Y\}$$
 which is a closed subgroup of $\op{N}(U)$ containing $U$.
We will show that $yS$ is locally closed;
this contradicts Lemma \ref{lem.normalizerorbit}.
We first claim that \be\label{yss} y\cal{O}\cap Y\subset yS.\ee
If $g\in \cal O$ such that $yg\in Y$, then $g\in \op{N}(U)$ by \eqref{eq.G-N}. Therefore
$Yg=\overline{yU} g=\overline{yg U}\subset Y$.
Moreover, since $Yg\subset Y \subset \Om_+$ and $\Om_+=\Om U$, we have $Yg\cap\Om\neq\emptyset$.
By the minimality assumption on $Y$, $Yg=Y$, proving that $g\in S$, and hence \eqref{yss}.

Therefore  $yS$ is an open  $U$-invariant subset of $Y$.
Since $Y=\cl{yS}$,  it follows that $yS$ is locally closed.
\end{proof}

By a one-parameter semigroup of $\cal H$, we mean a subset of the form $\{\exp (t\xi): t\ge 0\}$
for some non-zero $\xi$ in the Lie algebra of $\cal H$.
\begin{proposition}[Translate of $Y$ inside of $Y$]\label{prop.YSY} 
There exists a one-parameter subsemigroup $S<AMU_2$ such that $S\not\subset M$ and
$$YS\subset Y.$$
\end{proposition}
\begin{proof} It suffices to prove that there exists a sequence $\beta_k\to e$ in $AMU_2 -M$ such that $Y\beta_k\subset Y$
(cf. \cite[Lemma 10.5]{LO}).
Choose $y\in Y\cap\Om$.
By Lemma \ref{lem.G-N}, there exists $g_i\to e$ in $\cal H-\op{N}(U)$ such that $yg_i\in Y$. 
 Let $\cal O_k$ be a decreasing sequence of neighborhoods of $e$ in $G$ so that $\bigcap_k \cal O_k=\{e\}$.
Fix $k$. Applying Lemma \ref{lem.UgU} to the sequence $g_i^{-1}$, we get  $u_i'\in U$ and $u_i\in\mathsf{T}(y)=\{u\in U : y   u\in  \Om\}$ such that $u_i'  g_i^{-1} u_i$ converges to some element
$\alpha_k \in (AMU_2-M )\cap \cal O_k$.

Since $Y\cap \Om$ is compact, by passing to a subsequence,
 $yu_i$ converges to some $ y_k\in Y\cap\Om$ as $i\to\infty$. Hence as $i\to \infty$,
 $$yg_i(u_i')^{-1} = (yu_i) ( u_i' g_i^{-1} u_i)^{-1} \to y_k \alpha_k^{-1} \in Y.$$

Since $y_k\in Y\cap \Om$ and $\alpha_k\in \op{N}(U)$, it follows that $Y \alpha_k^{-1}\subset Y$.
It remains to set $\beta_k:=\alpha_k^{-1}$.
\end{proof}

\begin{proposition}[Translate of $Y$ inside of $X$] \label{prop.R3}\label{YvX}\label{prop.YvX}
 Suppose that
 there exists $y\in Y\cap \Om$ such that  $X-yH$ is not closed. Then there exists a non-trivial element $v\in U_2$ such that
\begin{equation*}
Yv\subset X.
\end{equation*}
\end{proposition}
\begin{proof}
By the hypothesis, there exists a sequence  $g_i\to e$ in $\cal H-H$ such that $yg_i\in X$. 
Since $X$ is $H$-invariant, we may assume $g_i\in H_2$.
Note that $\op{N}(U)\cap H_2=U_2$.
Hence if $g_i \in \op{N}(U)$ for some $i$, then we can simply take $v:=g_i$.

Now suppose that $g_i \notin \op{N}(U)$ for all $i$.
By Lemma \ref{lem.UgH}, there exists $u_i\in \mathsf T(y)$ such that $u_i^{-1}g_iu_i\to v$ for some non-trivial $v\in U_2$.
Observe
\begin{equation*}
(yu_i)(u_i^{-1}g_iu_i)=yg_iu_i\in X.
\end{equation*}
By passing to a subsequence, $yu_i$ converges to some $y_0\in Y\cap \Om$.  Since $y_0v\in X$,
it follows $Yv\subset X$ by the relative minimality of $Y$.
\end{proof}

\section{Expansion of an analytic curve inside a horospherical subgroup}

For $1\le k\le d-2$, let $\mS_k$ denote the collection
 of all $k$-dimensional spheres $S\subset \mathbb S^{d-1}$ such that $\Gamma_2 S$ is closed
in the space of all $k$-dimensional spheres of $\mathbb S^{d-1}$, and set 
$$\mS:=\bigcup_{1\le k\le d-2} \mS_k.$$
For each $1\le k\le d-2$,  there exists a connected reductive subgroup $L_k\simeq \SO^\circ(k+1,1)$
such that the convex hull of any sphere $S\subset \mathbb S^{d-1}$ of dimension $k$
 is equal to $\pi(g_SL_k)=\pi (g_S \op{N}(L_k)) $ for some $g_S\in H_2$, where $\pi: {H_2}  =\op{F}(\bH^d) \to \bH^d$ is 
the base-point projection. Moreover the space of $k$-dimensional spheres of $\bS^{d-1}$ is homeomorphic
to the quotient space $H_2/\op{N}(L_k)$. It follows that
 $S\in \mS_k$ if and only if $[g_S] \op{N}(L_k)$ is closed. We note that $\mS$ consists of countably many spheres (cf. \cite[Coro. 5.8]{LO}).

We deduce the following density statement from  the equidistribution result
\cite[Theorem 1.5]{Sh2}:
\begin{thm}\label{prop.CZ}\label{yang}
Let $\cal C:[0,1]\to U_2$ be a non-constant analytic curve.
Let $g_2\in H_2$ be such that $(g_2\cal C(0))^-\in \bS^{d-1}$ is not contained in any sphere in $\mS$. Then for any sequence $t_i\to +\infty$,
\begin{equation*}
\limsup_{i\to\infty}[g_2]\cal Ca_{-t_i}=Z_2.
\end{equation*}
\end{thm}
\begin{proof}
 Let $S$ be the smallest sphere of $\mathbb S^{d-1}$ which contains the subset $(g_2\op{Im}(\cal C))^-=\{(g_2 \mathcal C(s))^-: s\in [0,1]\}$.
 As $\cal C$ is non-constant,
the dimension $k$ of $S$ is at least $1$. Since $(g_2\cal C(0))^-\in S$, $S$  is not contained in any sphere in $\mS$ by the hypothesis on $(g_2\cal C(0))^-$.
Since $\cal C$ is non-constant analytic, $\{s\in[0,1] : \cal C'(s)=0\}$ is a finite set. Similarly, it follows from the hypothesis on $\cal C$ that
for any $S_0\in\mS$, the set $\{s\in[0,1] : (g_2\cal C(s))^-\in S_0\}$ is finite. 
Now the claim follows from the equidistribution theorem \cite[Theorem 1.5]{Sh2}. \end{proof}

\section{Invariance by analytic curves and conclusion}

\begin{thm}\label{lem.YCX}  \label{mmm} Let $X$ be a closed $H$-invariant subset of $Z$.
Let $Y\subset X$ be a $U$-minimal subset with respect to $\Om$.  Suppose that
 there exists $y\in Y\cap \Om$ such that  $X-yH$ is not closed. Then
there exists an analytic curve $\cal C: [0,1]\to U_2$ such that $\cal C'(0)\ne 0$ and
$$Y\cal C\subset X .$$
\end{thm}
\begin{proof}
By Proposition \ref{prop.YSY}, there exists a  one-parameter subsemigroup $S\subset MAU_2$  such that $S\not\subset M$ and $YS\subset Y$.
Now $S$ is either an unbounded subsemigroup of $w^{-1}MAw$ for some $w\in U_2$, or contained in $MU_2$ but not in $M$. 

\medskip

\noindent{{\bf Case 1}: $S\subset w^{-1}MAw$ for some $w\in U_2$ and $S$ is unbounded.}

\noindent{{\bf Case 1.a}:  $w=e$}. In this case, we have $S=\{(m_ta_t,m_ta_t):t\geq 0\}\subset MA$.
By Proposition \ref{prop.YvX}, there exists a nontrivial $v\in U_2$ such that $Yv\subset X$.
Observe $YSvAM\subset YvAM\subset X$.
Define $\cal C:[0,1]\to U_2$ by
\begin{equation*}
\cal C(t)=(e, m_ta_tva_t^{-1}m_t^{-1}).
\end{equation*}
Since $\cal C\subset SvAM$, we have $Y\cal C\subset X$.
If $\xi\in \op{Lie}(\cal A \cal M)$ such that $m_ta_t=\exp {t\xi}$, then
$\cal C(t)$ is given by $\op{Ad}_{\exp {t\xi}} v $ in the additive notation. Hence
$\cal C$ is analytic and $\cal C'(0)=\op{ad}_\xi (v) \ne 0$, since $\xi\notin \op{Lie}(\cal M)$.

\noindent{{\bf Case 1.b}:  $w\ne e$}.
We write $S=\{(m_ta_t,w^{-1}m_ta_tw):t\geq 0\}$.
Observe that $YSAM\subset X$, and define $\cal C:[0,1] \to U_2$ by
\begin{equation*}
\cal C(t)=(e,w^{-1}m_ta_twa_t^{-1}m_t^{-1}).
\end{equation*}
Since $\cal C\subset  SAM$, we have $Y\cal C\subset X$.
If $\xi\in \op{Lie}(\cal A \cal M)$ such that $m_ta_t=\exp {t\xi}$, then $\cal C(t)$ is given by $(\op{Ad}_{e^{t\xi}} w) -w$ in the additive notation. Hence
$\cal C$ is analytic and $\cal C'(0)=\op{ad}_\xi (w) \ne 0$, since $\xi\notin \op{Lie}(\cal M)$.

\medskip

\noindent{{\bf Case 2}: $S\subset MU_2$.}
Write $S=\{\exp (t(\xi+\eta)): t\ge 0\}$
where $\xi\in\op{Lie} M$ and $\eta\in\op{Lie} U_2-\{0\}$.
Define $\cal C : [0,1]\to U_2$ so that $\cal C(t)$ is the $U_2$-component of $\exp (t(\xi+\eta))$, which
 is explicitly given by $ \sum_{n=1}^\infty \frac{(-\xi)^{n-1}\eta}{n!} t^n$ in the additive notation.
So $\cal C(t)$ is analytic and $\cal C'(0)=\eta \ne 0$.
Since $\cal C\subset  SM$, we have $Y\cal C\subset X$.
\end{proof}

\begin{prop}  \label{fin} Let $E$ be an $H$-invariant subset of $Z$ which is not closed.
Then $E$ is dense in $Z$. \end{prop}
\begin{proof} Let $X$ denote the closure of $E$. By the assumption that
$E$ is not closed, there exists $x\in X-E$. Since any $H$-orbit meets $\Omega$,
we may assume $x\in (X-E)\cap \Omega$, by modifying $x$ using an element of $H$.

We claim that there exists a  $U$-minimal subset $Y$ of $X$ with respect to $\Om$ such that
for some $y\in Y\cap \Om$,  $X-yH$ is not closed.

If $E$ is locally closed, then $X-E$ is a closed subset. Let $Y$ be a $U$-minimal subset of $X-E$ with respect to $\Om$. Choose
$y\in Y\cap \Om$. Then $X-yH$ is not closed, since $y\in X-E$.

If $E$ is not locally closed, then $X-E$ is not closed. Let $Y$ be a $U$-minimal closed subset of $\cl{xH}$ with respect to $\Om$.
If $Y\cap \Om\subset xH$, choose $y\in Y\cap \Om$. 
If $Y\cap \Om\not\subset xH$, then choose $y\in (Y\cap \Om)-xH$.
We can then check that $X-yH$ is not closed.

Therefore, by Theorem \ref{lem.YCX}, there exists a non-constant analytic curve $\cal C:[0,1]\to  U_2$ such that 
$$Y\cal C\subset X.$$
By Theorem
\ref{prop.CZ},
there exists $y_2\in Z_2$ such that for any sequence $t_i\to  +\infty$,
\begin{equation}\label{eq.ED}
\limsup_{i\to\infty}y_2\cal Ca_{-t_i}=Z_2.
\end{equation}
 By Lemma \ref{min}, we
can choose $y_1\in \RFPM$ such that $(y_1,y_2)\in Y$.
Choose $g_i\in H_i$ so that $y_i=[g_i]$.

We claim that we can choose $\bt \in \br^{d-1}$ so that $(g_1 u_\bt)^-\in \Lambda_1$ and $(g_2u_\bt \cal C(0))^-$ does not belong {to} any sphere
contained in $\mS$.  It is convenient to use the upper-half space model of $\bH^{d}$ in which we have $\partial \bH^d=\br^{d-1}\cup\{\infty\}$, and
can take $v_o\in \T^1(\bH^d) $ {to} be the upward normal vector at $o=(0,\cdots, 0, 1)$ so that $v_o^+=\infty$ and $v_o^-=0$.
Then for all $\bt\in \br^{d-1}$,  we have $(u_\bt)^+=\{\infty\}$ and $(u_\bt)^-= u_\bt (0)=\bt$.
Suppose that the claim does not hold.  Then for any $\bt\in \br^{d-1}$ such that
 $\bt \in g_1^{-1}\Lambda_1 \cap \br^{d-1} $, we have $\bt  + \cal C(0) \in g_2^{-1}S\cap \br^{d-1}$ for some sphere $S\in \mS$.
Therefore, 
$ \Lambda_1 -g_1(\infty) \subset  \bigcup_{S\in \mS} g_1 \cal C(0)^{-1}g_2^{-1}S$.
This is a contradiction, since $\La_1$, being the limit set of a Zariski dense subgroup of $G$, cannot be contained in the union of countably many proper sub-spheres
of $\bS^{d-1}$ by \cite[Coro. 1.4]{FS}.

By replacing $(y_1,y_2)$ with $(y_1u_\bt,y_2u_\bt)$, we may now assume  that $y_1\in\RFM$, as \eqref{eq.ED} holds  for $y_2u_\bt$ as well
by Theorem \ref{yang}.

Since $y_1$ belongs to the compact $A_1$-invariant subset $\RFM$, there exists $t_i\to +\infty$ such that $y_1a_{-t_i}$
converges to some $ z_1\in \RFM$.
As $(y_1,y_2)\in Y$ and $X$ is $A$-invariant, it follows
\begin{equation*}
(y_1a_{-t_i} ,y_2\cal C a_{-t_i})\subset X.
\end{equation*}
By \eqref{eq.ED}, we obtain $\{z_1\}\times Z_2\subset X$.
Since $X$ is $H$-invariant, this implies $X=Z$.
\end{proof}

A collection of elements $g_1, \cdots, g_k\in \SO^\circ(d,1)$, $k\ge 2$, is called a Schottky generating set
if there exist mutually disjoint closed round balls $B_1, \cdots, B_k$ and $ B_1', \cdots, B_k'$ in $\bb S^{d-1}$
such that $g_i$ maps the exterior of $B_i$ onto the interior of $B_i'$ for each $i=1, \cdots, k$.
A subgroup of $\SO^\circ(d,1)$ is called a (classical) Schottky subgroup if it is generated by some Schottky generating set.
It is easy to see that a Schottky subgroup is a convex cocompact subgroup.

The following lemma is well-known (e.g., \cite[Proposition 4.3]{Be}). We give a short  elementary proof.

\begin{lem}\label{sh}
Any Zariski dense discrete subgroup $\Gamma$ of $\SO^\circ (d,1)$ contains a Zariski dense Schottky subgroup.
\end{lem}
\begin{proof} Let $\La$ denote the limit set of $\Gamma$.  For each hyperbolic element $\gamma\in \Gamma$,
$\gamma^{+}$ and $\gamma^-$ are respectively the attracting and repelling fixed points of $\gamma$.
As $\Gamma$ is non-elementary, it follows from \cite[Proposition 2.7]{Eb} that
the set $\{(\gamma^+, \gamma^-): \text{$\gamma $ is a hyperbolic element of $\Gamma$}\}$ is a dense subset of $\Lambda\times \Lambda$.

Choose two hyperbolic elements $\ga_1,\ga_2 \in \Ga$ such that $\{\ga_1^{\pm}\}$ and $\{\ga_2^{\pm}\}$ are disjoint
from each other. Let $S_1$ be the smallest sub-sphere of $\mathbb S^{d-1}$ which contains
 $\{\ga^{\pm}_i: i=1,2\}$. If $S_1\ne \mathbb S^{d-1}$, then we choose a hyperbolic element $\ga_3\in \Ga$ so that
 $\{\ga_3^{\pm}\}\cap S_1=\emptyset$. Let $S_2$ be the smallest sub-sphere of $\mathbb S^{d-1}$ which contains $\{\ga_{i}^{\pm}: i=1,2,3\}$.
 Then $\op{dim}S_2 >\op{dim} S_1$.
 Continuing in this fashion, we can find a sequence of hyperbolic elements $\ga_1, \cdots, \ga_m$ of $\Gamma$
 with $m\le d-1$ such that 
 the sets $\{\ga_i^{\pm}\}$ are all mutually disjoint and their union  is not contained in any proper sub-sphere of $\mathbb S^{d-1}$.
 Now for  a sufficiently large $k$, we can find pairwise disjoint round balls $B_i^{\pm}$ in $\bb{S}^{d-1}$
such that $\ga_i^{k}$ maps the exterior of $B_i^-$ to the interior of $B_i^+$ for each $i$; this is possible as $\ga_i^\pm$ are all distinct and for each $i$, $B_i^\pm$ can be chosen arbitrarily close to $\ga_i^\pm$ as we make $k$ large.
 Hence they form a Schottky generating set.
 Let  $\Gamma_0$ be the subgroup generated by them.
 Since the limit set of $\Gamma_0$ contains all fixed points of $\ga_i^k$, that is,
 $\{\ga_i^{\pm}:i=1, \cdots, m\}$, it is not contained in any proper sub-sphere of $\bb S^{d-1}$. Hence $\Gamma_0$ is Zariski dense.
\end{proof}

\noindent{\bf Proof of Theorems \ref{mt} and \ref{dm}.} 
In order to use the notations introduced in sections 2-6, 
let $\Gamma_1<H_1$ be a Zariski dense discrete subgroup and $\Gamma_2 $ be a cocompact lattice in $H_2$.
Since $\Gamma_1$ is countable and $\Gamma_2\ba H_2$ is compact, a closed $\Gamma_1$-orbit in $\Ga_2\ba H_2$ is necessarily finite.


 For any $(g_1, g_2)\in H_1\times H_2$, observe that the following are all equivalent to each other:
\begin{enumerate}
\item The orbit $[(g_1, g_2)]H $ is closed (resp. dense) in $( \G_1\times \Gamma_2 )\ba (H_1\times H_2)$;
\item The orbit $(\Gamma_1\times \Gamma_2) [(g_1, g_2)]$ is closed (resp. dense) in $ (H_1\times H_2) /H$;
\item The product $\Gamma_2 g_2 g_1^{-1} \Gamma_1$ is closed (resp. dense) in $G$;
\item The orbit $[g_2g_1^{-1}]\Gamma_1 $ is finite (resp. dense) in $ \G_2\ba H_2$. 
\end{enumerate}




We first claim Theorem \ref{dm} when  $\Gamma_1$ is convex cocompact.
Suppose that $X$ is a closed $H$-invariant subset of $Z=\Gamma_1\ba H_1\times \Gamma_2\ba H_2$, and suppose that $X\ne Z$.
If $X$ consists of finitely many $H$-orbits, then each of them must be closed by Proposition \ref{fin}.
Now suppose that $X$ contains infinitely many $H$-orbits, say $x_iH$. Each $x_iH$ should be closed again  by Proposition \ref{fin}.
Consider the set $E:=\bigcup x_iH$. Recalling that every $H$-orbit meets $\Om$,
 we may assume that $x_i\in \Om$ and it  converges to some $x\in \Om-E$; if $x\in x_jH$,
then we replace $E$ by $\bigcup_{i>j} x_iH$. Since $E$ is not closed, by Proposition \ref{fin}, $E$ is dense in $Z$. This proves the claim.
In view of the above equivalence,  Theorem \ref{mt} follows when $\Gamma_1$ is convex cocompact.

Since any Zariski dense discrete subgroup of $H_1$ contains a Zariski dense convex cocompact subgroup
by Lemma \ref{sh},  Theorem \ref{mt} follows. This implies Theorem \ref{dm} for a general Zariski dense discrete subgroup again in view
of the above equivalence.

\end{document}